\documentclass{article}
\usepackage[utf8]{inputenc}
\usepackage{hyperref}
\usepackage{amsmath}
\usepackage{amssymb}
\usepackage{amsthm}
\usepackage{graphicx}
\usepackage{subfig}
\usepackage{svg}
\usepackage{tikz}
\usepackage{tikz-cd}
\usetikzlibrary{intersections,calc}
\makeatletter
\DeclareRobustCommand{\cev}[1]{%
	{\mathpalette\do@cev{#1}}%
}
\newcommand{\do@cev}[2]{%
	\vbox{\offinterlineskip
		\sbox\z@{$\m@th#1 x$}%
		\ialign{##\cr
			\hidewidth\reflectbox{$\m@th#1\vec{}\mkern4mu$}\hidewidth\cr
			\noalign{\kern-\ht\z@}
			$\m@th#1#2$\cr
		}%
	}%
}
\makeatother
\usepackage{graphics}
\usepackage{tikz-cd}
\usepackage{enumerate}
\usepackage[mathscr]{euscript}

\usepackage{todonotes}

\theoremstyle{plain}
\newtheorem{Th}{Theorem}[section]
\newtheorem{Lemma}[Th]{Lemma}

\theoremstyle{definition}
\newtheorem{Def}[Th]{Definition}

\newtheorem{Rem}[Th]{Remark}
\newtheorem{?}[Th]{Problem}
\newtheorem{Ex}[Th]{Example}

\newcommand{\R}{\mathbb{R}}

\newcommand{\C}{\mathbb{C}}

\newcommand{\holr}{\mathrm{Hol}^{\mathrm{reg}}}

\DeclareMathOperator{\id}{id}

\DeclareSymbolFont{rsfs}{U}{rsfs}{m}{n}
\DeclareSymbolFontAlphabet{\mathscrsfs}{rsfs}

\title{Generalized Pentagon Equations}
\author{Anton Alekseev, Florian Naef, Muze Ren}

\begin{document}
\date{}
\maketitle
\begin{abstract}
Drinfeld defined the Knizhinik--Zamolodchikov (KZ) associator $\Phi_{\rm KZ}$ by considering the regularized holonomy
of the KZ connection along the {\em droit chemin}  $[0,1]$.  The KZ associator is a group-like element of the free associative algebra with two generators, and it satisfies the pentagon equation. 

In this paper, we consider paths on $\mathbb{C}\backslash \{ z_1, \dots, z_n\}$ which start and end at tangential base points. These paths are not necessarily straight, and they may have a finite number of transversal self-intersections.
We show that the regularized holonomy $H$ of the KZ connection associated to such a path satisfies a generalization of Drinfeld's pentagon equation. In this equation, we encounter $H$, $\Phi_{\rm KZ}$, and new factors associated to self-intersections, to tangential base points, and to the rotation number of the path.
\end{abstract}

\section{Introduction}
The pentagon equation was first introduced in the theory of braided monoidal categories by Mac Lane \cite{Maclane1998}. 
In a braided monoidal category $\mathcal{C}$ with tensor product $\otimes$, braiding $b$ and associativity  constraint $a$
one has the following commuting pentagon diagram:
\begin{equation}\label{eq:cate_Pentagon}
\begin{tikzcd}
&(W\otimes X)\otimes (Y\otimes Z)\arrow[rd,"a_{W,X,Y\otimes Z}"]&\\
\big((W\otimes X)\otimes Y\big)\otimes Z\arrow[ru,"a_{W\otimes X,Y,Z}"]\arrow[d,"a_{W,X,Y}\otimes\id"]&&W\otimes\big(X\otimes (Y\otimes Z)\big)\\
\big(W\otimes (X\otimes Y)\big)\otimes Z\arrow[rr,"a_{W,X\otimes Y,Z}"]&&W\otimes\big((X\otimes Y)\otimes Z\big)\arrow[u,"\id_{W}\otimes a_{X,Y,Z}"]\\
\end{tikzcd}
\end{equation}
and two hexagon diagrams.

%One of the first appearance of pentagon equation \eqref{eq:cate_Pentagon} was in the theory of monoidal category by Mac Lane \cite{Maclane1998}.
%A monoidal category is a category $\mathcal{C}$, with the tensor product $\otimes$, the associativity  constraint $a$, the distinguished unit object $I$, the left unit constraint $l$ and right unit constraint $r$, it needs to satisfy the pentagon equation \eqref{eq:cate_Pentagon} for $a$ and two triangle diagrams \eqref{eq:cate_triangle} for $l$ and $r$. 
%\begin{equation}\label{eq:cate_Pentagon}
%\begin{tikzcd}
%&(W\otimes X)\otimes (Y\otimes Z)\arrow[rd,"a_{W,X,Y\otimes Z}"]&\\
%\big((W\otimes X)\otimes Y\big)\otimes Z\arrow[ru,"a_{W\otimes X,Y,Z}"]\arrow[d,"a_{W,X,Y}\otimes\id"]&&W\otimes\big(X\otimes (Y\otimes Z)\big)\\
%\big(W\otimes (X\otimes Y)\big)\otimes Z\arrow[rr,"a_{W,X\otimes Y,Z}"]&&W\otimes\big((X\otimes Y)\otimes Z\big)\arrow[u,"\id_{W}\otimes a_{X,Y,Z}"]\\
%\end{tikzcd}
%\end{equation}

%\begin{equation}\label{eq:cate_triangle}
%\begin{tikzcd}
%(I\otimes V)\otimes W\arrow[r,"a_{I,V,W}"]\arrow[d,"l_V\otimes\id_{W}"]&I\otimes (V\otimes W)\arrow[ld,"l_{V\otimes W}"]\\
%V\otimes W&
%\end{tikzcd} 
%\begin{tikzcd}
% (V\otimes W)\otimes I\arrow[r,"a_{V,W,I}"]\arrow[rd,"l_{V\otimes W}"]&V\otimes (W\otimes I)\arrow[d,"\id_{V}\otimes r_{W}"]\\
% &V\otimes W
%\end{tikzcd}
%\end{equation}

The Mac Lane Coherence Theorem \ref{th:coherence} (originally in \cite{Maclane1998}, we cite the form in \cite{Etingof2015}) explains relation between morphisms in braided monoidal categories and elements of the braid group.

%It provides the equivalence of arbitrary parentheses of tensor product of arbitrary number ($\ge 3$) of objects. 
%
\begin{Th}[{\cite[Theorem~2.9.2]{Etingof2015}}]\label{th:coherence}
Let $X_1,\dots,X_n\in \mathcal{C}$, $P_1,P_2$ be two parenthesized products of $X_1,\dots,X_n$ (in any order order) possibly with insertions of the unit object $I$, and $f,g: P_1\to P_2$ be two isomorphisms obtained by composing braiding, associativity and unit isomorphism and their inverses possibly tensored with identity morphisms. If the underlying braids are isotopic, then $f=g$.
\end{Th}

Let $\mathfrak{g}$ be a complex quadratic Lie algebra (that is, a finite dimensional Lie algebra which possesses an invariant scalar product). Drinfeld \cite{Drinfeld1989,Drinfeld1990} turned the category $U\mathfrak{g}-{\rm mod}_\hbar$ of its finite dimensional representations into a braided monoidal category using solutions of the Knizhnik-Zamolodchikov (KZ) equations.

%{\color{blue}Drinfeld introduced the theory of quasi-Hopf algebra \cite{Drinfeld1989,Drinfeld1990} as the algebraic counter part of the monoidal category, namely the modules of the quais-Hopf algebra forms a monoidal category. He then constructed the universal or "standard" quasi-triangular quasi-Hopf quantized universal enveloping (QUE) algebra in his paper by using the Knizhinik-Zamolodchikov connection. The universality is in the following sense \ref{th:universal}.
%\begin{Th}[\cite{Drinfeld1989},Theorem B]\label{th:universal} Any quasitriangular quasi-Hopf QUE algebra can be made standard by a suitable twist.
%\end{Th}

%}
In more detail, the KZ connection for $n$ points on the complex plane 
\begin{equation}
A_{\rm KZ} = \frac{1}{2\pi i} \sum_{i<j} t_{i,j} \, d\log(z_i - z_j).
\end{equation}
takes values in the Drinfeld-Kohno Lie algebra of infinitesimal braids $\mathfrak{t}_{n}$ with generators $t_{ij}=t_{ji}$ with $i,j \in \{ 1, \dots, n\}$ and the relations
\begin{equation} \label{eq:DK}
[t_{i,j}, t_{k,l}]=0, \hskip 0.3cm [t_{i,j} + t_{i,k}, t_{j,k}]=0
\end{equation}
for distinct indices $i,j,k,l$.

The braiding constraints are given by $b_{i,j} = \exp(t_{i,j}/2)$, and the associativity constraint is defined by the regularized holonomy of the KZ equation for $n=3, z_1=0, z_3=1$ and $z=z_2$ moving on the straight segment ({\em droit chemin}) $[0,1]$. It carries the name of the KZ associator:
$$
\Phi_{\rm KZ}(t_{12}, t_{23}) ={\rm Hol}^{\rm reg}\left( \frac{t_{12}}{2\pi i} \, d\log(z) + \frac{t_{23}}{2\pi i} \, d\log(z-1), [0,1]\right).
$$
The KZ associator satisfies the pentagon equation in $U(\mathfrak{t}_4)$:
\begin{equation}\label{eq:Drinfeld_Pentagon}
\Phi_{\rm{KZ}}(t_{23}, t_{34})\Phi_{\rm{KZ}}(t_{1,23}, t_{23,4})\Phi_{\rm{KZ}}(t_{12}, t_{23})=\Phi_{\rm{KZ}}(t_{12}, t_{2,34})\Phi_{\rm{KZ}}(t_{12,3}, t_{3,4}),
\end{equation}
where $t_{i,jk}=t_{ij}+t_{ik}$.

%for a point 
%$z$ moving on the complex plane with $n$ marked points $\mathbb{C}\backslash \{ z_1, \dots, z_n\}$ takes values in the Drinfeld-Kohno Lie algebra of infinitesimal braids $\mathfrak{t}_{n+1}$ with generators $t_{ij}=t_{ji}, t_{iz}$ with $i,j \in \{ 1, \dots, n\}$:
%
%\begin{equation}\label{eq:intro_KZ}
%A_{z}=\frac{1}{2\pi i}\sum_{i=1}^{n} t_{i,z} \, d\log(z-z_i).
%\end{equation}

%The case of $n=2, z_1=0, z_2=1$ attracted a lot of attention (see {\em e.g.} %\cite{Deligne1970,Deligne1989}).
%In that case, the regularized holonomy associated to the straight segment ({\em droit chemin}) $[0,1]$ is called the KZ associator. It is defined as follows,

In this paper, we generalize equation \eqref{eq:Drinfeld_Pentagon} to the case of an arbitrary number of marked points on the complex plane, and to paths which are not necessarily straight and which may have self-intersections. More precisely, we will consider the KZ equation on $\Sigma=\mathbb{C}\backslash \{ z_1, \dots, z_n\}$. For a point $z$ moving on $\Sigma$ we associate the KZ connection
\begin{equation}\label{eq:intro_KZ}
A_{z}=\frac{1}{2\pi i}\sum_{i=1}^{n} t_{i,z} \, d\log(z-z_i)
\end{equation}
taking values in $\mathfrak{t}_{n+1}$.
%We want to generalize the pentagon equation in the spirit of Drinfeld \cite{Drinfeld1989, Drinfeld1990} by considering arbitrary regular homotopy class of curves possibly with self intersections.
Consider a path $\gamma$ which  starts at a tangential base point $(z_i, v_i)$ and ends at another tangential base point $(z_j, v_j)$, where $v_i, v_j$ are nonvanishing  vectors at $z_i$ and $z_j$, respectively. Denote by $\rm{rot}(\gamma)$
the rotation number of $\gamma$ with respect to the blackboard framing. We assume that $\gamma$ has a finite number of transverse self-intersection points. To each self-intersection point, one can associate the intersection number $\varepsilon_l\in \{1,-1\}, l=1, \dots, m$ indicating the orientation of the frame formed by two tangent vectors to the path at that point.

The main object of our study is the regularized holonomy
$$
H = {\rm Hol}^{\rm reg}(A_z, \gamma)
$$
associated to the path $\gamma$.
By adding one more point $w$ to the configuration, we can introduce several versions of the holonomy $H$ needed in the statement of the generalized pentagon equation:
\begin{subequations}\label{eq:intro_honolomies}
\begin{equation}
H_z:=\holr\left( \frac{t_{z,wj}}{2\pi i}\, d\log(z-z_j)+\sum_{k\ne j}\frac{t_{k,z}}{2\pi i}\, d\log(z-z_k),\gamma\right),   
\end{equation}
\begin{equation}
H_w:=\holr\left( \frac{t_{iz,w}}{2\pi i} \, d\log(w-z_i)+\sum_{k\ne i}\frac{t_{k,w}}{2\pi i}\, d\log(w-z_k),\gamma\right),
\end{equation}
\begin{equation}
H_{zw}:=\holr\left( \sum_{k=1}^n \frac{t_{k,zw}}{2\pi i} \, d\log(z-z_l),\gamma\right).
\end{equation}
\end{subequations}

Note that all these versions are obtained from $H$ by various substitutions, namely if we write
\[
H = H(t_{1,z}, \dots, t_{n,z})
\]
we have
\begin{subequations}\label{eq:intro_honolomiessubst}
\begin{equation}
H_z= H(t_{1,z}, \dots, t_{(j-1),z}, t_{jw,z}, t_{(j+1),z}, \dots, t_{n,z}),
\end{equation}
\begin{equation}
H_w=H(t_{1,w}, \dots, t_{(i-1),w}, t_{iz,w}, t_{(i+1),w}, \dots, t_{n,w}),
\end{equation}
\begin{equation}
H_{zw}=H(t_{1,zw}, \dots, t_{n,zw}).
\end{equation}
\end{subequations}

Equipped with this notation, we can now state (a somewhat simplified version of) the generalized pentagon equation:
\begin{Th}      \label{thm:main}
	There exist elements $C_l\in \exp(\mathfrak{t}_{n+2})$ for $l=1, \dots, m$ such that the following identity holds in $U(\mathfrak{t}_{n+2})$,

 \begin{equation} \label{eq:general_pentagon}
 \begin{array}{ll}
\Phi_{\rm{KZ}}(t_{zw},t_{wj}) \left(H_{zw} |v_j/v_i|^{t_{zw}/2\pi i} e^{\text{rot}(\gamma)t_{zw}}\right) \Phi_{\rm{KZ}}(t_{iz},t_{zw}) & = \\
H_z \left(\prod_{l=1}^{m}C_{l}^{-1}e^{-\varepsilon_l t_{zw}}C_{l}\right) H_w.
\end{array}
 \end{equation}	
\end{Th}

% {\color{red}Moreover we have the following property of the holonomies $C_l$, it characterize $C_l$ under the projection $\pi$ from $\mathfrak{t}_{n+2}\to \mathfrak{t}_{n+2}/\langle t_{zw}\rangle$, we have 

% \begin{Th}[same as Theorem \ref{thm:C_l}]
%     We have,
%     %
%     \begin{equation}
%         \pi(C_l) = {\rm Hol}^{\rm reg}(\pi(A_z), \gamma_{[0,t_l]}) 
%         {\rm Hol}^{\rm reg}(\pi(A_w), \gamma_{[1,s_l]}).
%     \end{equation}
% \end{Th}

% }

For the case of $n=2, z_1=0, z_2=1$ and $\gamma=[0,1]$, we recover Drinfeld's pentagon equation \eqref{eq:Drinfeld_Pentagon}. In the general case, equation \eqref{eq:general_pentagon} contains the holonomy $H$, the KZ associator $\Phi_{\rm KZ}$, as well as new terms $|v_j/v_i|^{t_{zw}/2\pi i}$ and $e^{\text{rot}(\gamma)t_{z,w}}$ depending on the tangential base points and on the rotation number of the path $\gamma$, as well as $(C_{l}^{-1}e^{-\varepsilon_l t_{z,w}}C_{l})$ which correspond to self-intersections
\[
\gamma(t_l) = \gamma(s_l), \quad 0 \leq t_l < s_l \leq 1
\]
of $\gamma$. We obtain the following partial information on $C_{l}$.
\begin{Th}[see Theorem \ref{thm:C_l}]
    We have,
    \begin{equation}
        C_l = {\rm Hol}^{\rm reg}(A_z, \gamma_{[0,t_l]}) 
        {\rm Hol}^{\rm reg}(A_w, \gamma_{[1,s_l]}) + O(t_{zw})
    \end{equation}
\end{Th}

% $(C_{l}^{-1}e^{-\varepsilon_l t_{z,w}}C_{l})$ corresponding to self-intersections of $\gamma$, and the expressions $|v_j/v_i|^{t_{zw}/2\pi i}$ and $e^{\text{rot}(\gamma)t_{z,w}}$ which depend on the tangential base points and on the rotation number of the path $\gamma$.

Note that Theorem \ref{thm:main} opens the following new perspective on the topic: one can use the full KZ connection for $n+2$ points to bring the configuration $\{ z_1, \dots, z_n\}$ to a tangential base point. Then, by Mac Lane Coherence Theorem equation \eqref{eq:general_pentagon} should reduce to some equality of braids\footnote{We are grateful to D. Bar-Natan for this remark.} (and eventually to a number of pentagon and hexagon constraints). This new interpretation should include combinatorial expressions for the rotation number, indices of intersection points and explicit presentations of the elements $C_l$. It may possibly be related to the work of G. Massuyeau \cite{Massuyeau} in which a similar geometric setup was considered for applications to the Goldman-Turaev Lie bialgebra.
We hope to return to this interesting point of view in future work.

In a forthcoming work, we plan to apply Theorem \ref{thm:main} for computing van den Bergh double brackets of universal regularized holonomies ${\rm Hol}^{\rm reg}(A_z, \gamma)$ (with values in a free associative algebra), and of Poisson brackets of the corresponding regularized holonomies for the Lie algebra $\mathfrak{g}={\rm gl}(N, \mathbb{C})$.

The structure of the paper is as follows. In Section \ref{sub:local_KZ}, we discuss local solutions of the KZ equation near regular and tangential base points. In Section \ref{sub:holonomy_maps}, we define regularized holonomies, and in Section \ref{sub:composition} we recall their basic property under composition of paths.
 In Section \ref{sub:regions_local_solution}, we follow Drinfeld and construct local solutions of the KZ equation near tangential base points.
 In Section \ref{sub:self_intersection_contributions}, we analyse solutions of KZ equations self-intersection points of the path and prove the generalized Pentagon equation. Finally in subsection \ref{sub:property_C_l}, we discuss the property of the terms $C_l$ in the equation.
 
 \vskip 0.2cm
 
 {\bf Acknowledgements.} We are grateful to D. Bar-Natan, F. Brown and B. Enriquez for interesting discussions. 
 Reserach of A.A. and M.R. was supported in part by the grants number 208235 and 200400 and by the National Center for Competence in Research (NCCR) SwissMAP of the Swiss National Science Foundation.
 Research of A.A. was supported in part  by the award of the Simons Foundation to the Hamilton Mathematics Institute of the Trinity College Dublin under the program ``Targeted Grants to Institutes''.

\section{Regularized holonomy}\label{sec:regularized_holonomy}

In this Section, we define regularized holonomies of the KZ equation and recall their basic properties.

\subsection{Local solutions of the KZ equation} \label{sub:local_KZ}

Let $\mathfrak{t}_{n+1}$ be the Drinfeld-Kohno Lie algebra in $n+1$ strands. 
%In the following we will use the indices $\{ 1, \dots, n, z\}$. That is, 
Its generators are $t_{i,j}$ for $i,j \in \{ 1, \dots, n, z \}$ and $z$ stands for $n+1$,
and the relations are given by \eqref{eq:DK}.
Recall that the generators
$t_{i,z}$ for $i=1, \dots, n$ span a free Lie subalgebra $\mathfrak{f}_n$ of $\mathfrak{t}_{n+1}$.
\begin{Rem}
Notice that the symbol $z$ denotes both the last strand and the (coordinate of the) point moving on the surface.
\end{Rem}

For the surface $\Sigma = \mathbb{C}\backslash \{ z_1, \dots, z_n \}$, we consider 
%the following (part of) the Knizhnik--Zamolodchikov (KZ) connection:
%\begin{equation}
%A_{z}=\frac{1}{2\pi i}\sum_{i=1}^{n} t_{i,z}d\log(z-z_i)
%\end{equation}
the differential equation 
\begin{equation}\label{eq:KZ}
d\Psi=A_z\Psi,
\end{equation}
where $A_z$ is the flat connection given by \eqref{eq:intro_KZ} and 
 solutions $\Psi(z) \in U\mathfrak{f}_n \cong \mathbb{C}\langle\langle t_{1,z}, \dots, t_{n,z}\rangle\rangle$ take values in the (completion of) the universal enveloping algebra of $\mathfrak{f}_n$.

We will associate local solutions of the equation \eqref{eq:KZ} to regular points $p \in \mathbb{C}\backslash \{ z_1, \dots, z_n\}$ and to tangential base points $(z_i, v_i)$, where $0 \neq v_i \in \mathbb{C}$ is a tangent vector at $z_i$. To a regular point $p$, we associate the unique local solution $\Psi_p(z)$  which satisfies the initial condition $\Psi_p(p)=1$.

For a tangential base point $(z_i, v_i)$, let $B_\varepsilon(z_i)$ be a small open disk of radius $\varepsilon$ around $z_i$, and let 
$l_i=\{ z_i - sv_i; s\in \mathbb{R}_{\geq 0}\}$ be a ray emanating from $z_i$ in the direction opposite to $v_i$. 
We denote by $D_i=B_\varepsilon(z_i)\backslash l_i$ the simply connected domain obtained by deleting the ray $l_i$ from the disk $B_\epsilon(z_i)$.

\begin{Lemma}  \label{lem:solutions}
For $\varepsilon$  sufficiently small, there is a unique solution $\Psi_{z_i, v_i}(z)$ of the differential equation \eqref{eq:KZ} on the domain $D_i$ such that
\begin{equation}
\Psi_{z_i,v_i}(z)=f\left(\frac{z-z_i}{v_i}\right)\exp\left(\frac{t_{i,z}}{2\pi  i}\log\left(\frac{z-z_i}{v_i}\right)\right),
\end{equation}
where $f$ is an analytic function on $B_\varepsilon(0)$ with $f(0)=1$.
%and $\Psi_{z_i,v_i}(z)$ is group like in $U\mathfrak{f}_n$. 
\end{Lemma}

\begin{Rem}
Solutions $\Psi_{z_i, v_i}(z)$ have the following asymptotic behaviour for $z \to z_i$ in the domain $D_i$:
$$
\Psi_{z_i, v_i}(z) \sim_{z \to z_i}\left(1 + O\left(\frac{z-z_i}{v_i}\right) \right) 
\exp\left(\frac{t_{i,z}}{2\pi  i}\log\left(\frac{z-z_i}{v_i}\right)\right).
$$
This can also be used as a definition of these solutions.
\end{Rem}

\begin{proof}
For convenience of the reader, we recall the proof. Let $w=\frac{z-z_i}{v_i}$. By plugging the expression $f(w)w^{t_{iz}/2\pi i}$ into equation \eqref{eq:KZ}, we get
\begin{align}
v_i^{-1}\left(f'(w)+f(w) \frac{t_{i,z}}{2\pi i w}\right)=
v_i^{-1} \frac{t_{i, z}}{2\pi i w} f(w) + \sum_{j \neq i}
\frac{t_{j, z}}{2\pi i (z-z_j)} f(w).
\end{align}
This yields
$$
f'(w) - \frac{1}{w} \left[\frac{t_{i, z}}{2\pi i}, f(w)\right] = g(w),
$$
where $g(w)=\sum_k g_k w^k$ is an analytic function near zero. 

We are looking for a solution in the form $f(w)=\sum_k f_k w^k$ with $f_0=1$. 
Since $[t_{i,z},f_0]=[t_{i, z}, 1]=0$, 
the expression $\frac{1}{w} \left[\frac{t_{i, z}}{2\pi i}, f(w)\right]$ is actually regular at $w=0$. Then, we find the coefficients $f_k$ for $k \geq 1$:
$$
f_k = (k - {\rm ad}_{t_{i,z}/2\pi i})^{-1}g_{k-1}.
$$
This implies that the convergence radius of the power series $f$ is greater or equal to the one of the power series $g$, and hence $f$ defines an analytic function near the origin, as required.
\end{proof}

%It is now easy to see that the equation above admits a unique analytic solution with $f(0)=1$ on a sufficiently small disk near zero.
\begin{Lemma}     \label{lem:Psi_grouplike}
    Local solutions $\Psi_p(z)$ and $\Psi_{z_i, v_i}(z)$ are group like elements of the (completed) universal enveloping algebra $U\mathfrak{f}_n$.
\end{Lemma}

\begin{proof}
For solutions at  regular points, observe that the KZ connection is linear in generators of the Lie algebra $\mathfrak{t}_{n+1}$. Hence, $\Delta(A_z) = A_z \otimes 1 + 1 \otimes A_z$, where $\Delta$ is the standard coproduct. This implies that  $\Delta(\Psi_p)$ and $\Psi_p \otimes \Psi_p$ satisfy the same differential equation 
$$
d \hat{\Psi}(z) = (A_z \otimes 1 + 1 \otimes A_z) \hat{\Psi}(z)
$$
with initial condition $\hat{\Psi}(p)=1 \otimes 1$. Then, by the uniqueness property of solutions of first order ODEs we obtain $\Delta(\Psi_p)=\Psi_p \otimes \Psi_p$, as required.

Similarly, for a local solutions $\Psi_{z_i,v_i}$ we have
\begin{equation}
\Psi_{z_i,v_i}(z)=f\left(\frac{z-z_i}{v_i}\right)\exp\left(\frac{t_{i,z}}{2\pi  i}\log\left(\frac{z-z_i}{v_i}\right)\right),
\end{equation}
and we observe that $\Delta(f(w))$ and $f(w)\otimes f(w)$ satisfy the same differential equation
\begin{align}\label{eq-coproduct}
\hat{f}'(w)- \frac{1}{w} \left[\frac{t_{i, z}\otimes 1+1\otimes t_{i,z}}{2\pi i}, \hat{f}(w)\right] =\sum_{j \neq i}
\frac{v_i(t_{j, z}\otimes 1+1\otimes t_{j, z})}{2\pi i (z-z_j)} \hat{f}(w).
\end{align}
An argument similar to the proof of existence of solutions implies 
 uniqueness of an analytic solution of this equation with the initial value $\hat{f}(0)=1$. Hence, we can conclude as above.

% By an argument similar to the existence, this equation has a unique analytic solution with the initial value $h(0)=1$. 
% Hence, we have $\Delta(f(w))=f(w)\otimes f(w)$, and we conclude that $\Psi_{z_i,v_i}(z)$ is group like. 

\end{proof}

\subsection{Holonomy maps}       \label{sub:holonomy_maps}

Let $\gamma: [0,1] \to \mathbb{C}$ be a smooth path with $\dot{\gamma}= \tfrac{d\gamma}{dt} \neq 0$ which may start and end at regular or marked points but with $\gamma(t) \in \Sigma =\mathbb{C}\backslash \{ z_1, \dots, z_n\}$ for all $0<t<1$.
In case when $\gamma$ starts or ends at marked points (or both), we require that near that point $\gamma(t)$ be linear in $t$:
$$
\begin{array}{ll}
\gamma(t) = z_i + v_i t & {\rm for} \,\, 0< t < \varepsilon; \\
\gamma(t) = z_j + v_j(1-t) & {\rm for} \,\, 0< 1-t < \varepsilon.
\end{array}
$$
Here $(z_i, v_i)$ is the tangential starting point of $\gamma$, $(z_j, v_j)$ is the tangential end point of $\gamma$, and $\varepsilon$ is small enough. The path $\gamma(t)$ may have a finite number of transverse self-intersection points, $\gamma(s_l) = \gamma(t_l) = a_l$ for $l=1, \dots, m$. Then, we also require that $\gamma(t)$ be a linear function near the points $s_l$ and $t_l$.

To each path $\gamma(t)$, we associate two local solutions $\Psi_p(z)$ and $\Psi_q(z)$ corresponding to its end points (regular or tangential). For regular end points, we have
$p=\gamma(0), q = \gamma(1)$, and for tangential end points we have
$(p=(\gamma(0) ,\dot{\gamma}(0)), q = (\gamma(1), -\dot{\gamma}(1))$.
The solutions $\Psi_p(z)$ and $\Psi_q(z)$ can be analytically continued to a strip around the path $\gamma$. 
More precisely, we extend $\gamma \colon (0,1) \to \Sigma$ to a local diffeomorphism $\tilde{\gamma} \colon (0,1) \times (-\varepsilon, \varepsilon) \to \Sigma$. Note that the local solution $\Psi_p(z)$ defines a flat section of the pulled back connection $\tilde{\gamma}^*A_z$ on a neighborhood of $(0,0) \in [0,1) \times (-\varepsilon,\varepsilon)$ which can then be analytically continued to a solution on the entire strip. The same applies to the local solution $\Psi_q(z)$. By abuse of notation, we denote analytic continuations of local solutions by the same symbols. Note that if the path $\gamma$ has self-intersections, the self-intersection point has two pre-images in the strip $(0,1) \times (-\varepsilon, \varepsilon)$, and the values of analytically continued solutions at these points are different, in general.

\begin{Def}
For a path $\gamma$, the (regularized) holonomy of $A_z$ along $\gamma$
is given by
\begin{equation}\label{eq:def_regularized_hol}
\holr(A_z,\gamma):=\Psi_q^{-1}(z)\Psi_p(z)
\in U\mathfrak{f}_n,
\end{equation}
where $z$ is any point in the strip around $\gamma$. 
\end{Def}

When both end points of $\gamma$ are regular, this definition coincides with the standard definition of a parallel transport of a connection (on a trivial bundle). In particular, in this case the holonomy is invariant under homotopies which preserve end points of the path. Similarly, for tangential end points the regularized holonomy is invariant under homotopies which preserve $(\gamma(0), \dot{\gamma}(0))$ for the starting point and $(\gamma(1), -\dot{\gamma}(1))$ for the end point.

%In either case we can associate local solutions $\Psi^\gamma_s$ and $\Psi^\gamma_t$ to the start and the end point of $\gamma$. For instance, if $p=\gamma(0)$ is a regular point we set
%\[
%\Psi^\gamma_s := \Psi_p
%\]
%to it and if the $(\gamma(0), \dot{\gamma}(0)) = (z_i, v_i)$ we set 
%\[
%\Psi^\gamma_s := \Psi_{(z_i,v_i)}.
%\]
%We similarly define $\Psi^\gamma_t$ in the two cases.

%\begin{Ex}          \label{Ex:n=1}
%Consider $\Sigma = \mathbb{C}\backslash \{ 0 \}$ with the tangential base point %$(0,1)$. The corresponding KZ equation reads
%
%$$
%d \Psi = \frac{t_{1,z}}{2\pi i} \, \Psi,
%$$
%and it admits a local solution
%
%$$
%\Psi(z ) = \exp\left( \frac{t_{1,z}}{2\pi i} \, \log(z) \right).
%$$
%Choosing the branch cut of $\log(z)$ on the negative real axis, we have
%
%$$
%\Psi(x -i0)^{-1} \Psi(x+i0) = \exp(t_{1,z})
%$$
%for $x<0$.
%One can actually view $\Psi(x+i0)$ as a value of the solution $\Psi(z)$ analytically continued from the point $x-i0$ along a path going around the origin in the counter-clockwise direction. Furthermore, choosing two regular base points $p=\rho e^{i \theta_p},
%q = \rho e^{i\theta_q}$ and choosing a path between them going along the unit circle, we obtain
%
%$$
%{\rm Hol}_{p,q} = \exp\left( \frac{\theta_p - \theta_q}{2\pi} \, t_{1,z}\right).
%$$
   
%\end{Ex}        

\begin{Ex}      \label{ex:associator}
Consider $\Sigma = \C \setminus \{ 0 ,  1 \}$, tangential base points $(0, 1)$ and $(1, -1)$,
and $\gamma(t)=t$ the straight segment ({\em le droit chemin}) connecting $0$ to $1$.
% We consider the case where in the plane, we have only two marked points $0,1$ with tangent vector $1,-1$. 
We have two local solutions with asymptotic behavior
$$
\Psi_{0,1}(z)\sim z^{\frac{t_{1,z}}{2\pi i}}, \hskip 0.3cm \Psi_{1,-1}(z)\sim (1-z)^{\frac{t_{2,z}}{2\pi i}}.
$$
In this case, $\holr(A_z,[0,1])$ agrees with Drinfeld's definition of the KZ associator
$$
\Phi_{\rm KZ}(t_{1,z},t_{2,z})=\Psi_{1,-1}(z)^{-1} \Psi_{0,1}(z).
$$
\end{Ex}

\begin{Lemma}
   (Regularized) holonomies are group-like elements of the (completed) universal enveloping algebra $U\mathfrak{f}_n$.
\end{Lemma}

\begin{proof}
By Lemma \ref{lem:Psi_grouplike}, local solutions are group like. This property is preserved by analytic continuations. Since holonomies are given by ratios of two analytically continued local solutions, they are group like as well.
\end{proof}

\subsection{Composition of paths and rotation number} \label{sub:composition}

Two paths $\gamma_1$ and $\gamma_2$ with end points (regular or tangential) $(p_1, q_1)$ and $(p_2, q_2)$ are called composable if $p_2 = q_1$. In that case, we can define a composition of paths $\gamma_2 \gamma_1$. 

If the common end point $p_2 = q_1$ is a regular point, the composition is defined up to homotopy by concatenation of paths (this allows to make the resulting path smooth). If the common end point is a marked point, the composition is defined modulo regular homotopy (that is, the first Reidemeister move is not allowed). In that case, there are two natural ways to define the composition. One considers a neighborhood of the marked point (see Fig. \ref{fig:composing3}), and adds a clockwise or a counter-clockwise half-turn with respect to the blackboard orientation of $\mathbb{C}$ (see Fig. \ref{fig:composing}). In this paper, we chose to use the clockwise convention.

\begin{figure}[h]
    \centering    \includegraphics[width=0.6\columnwidth]{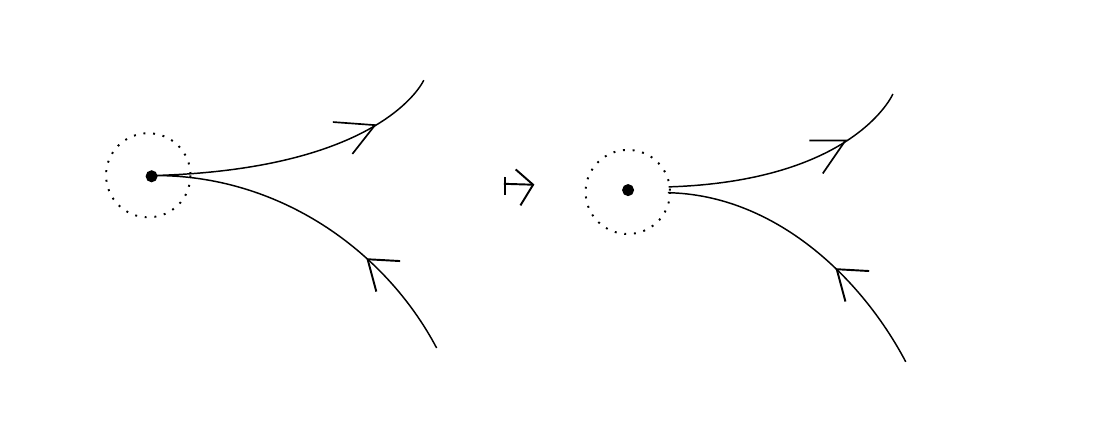}
    \caption{small region near the marked point}
    \label{fig:composing3}
\end{figure}

\begin{figure}[h]
    \centering   \includegraphics[width=0.6\columnwidth]{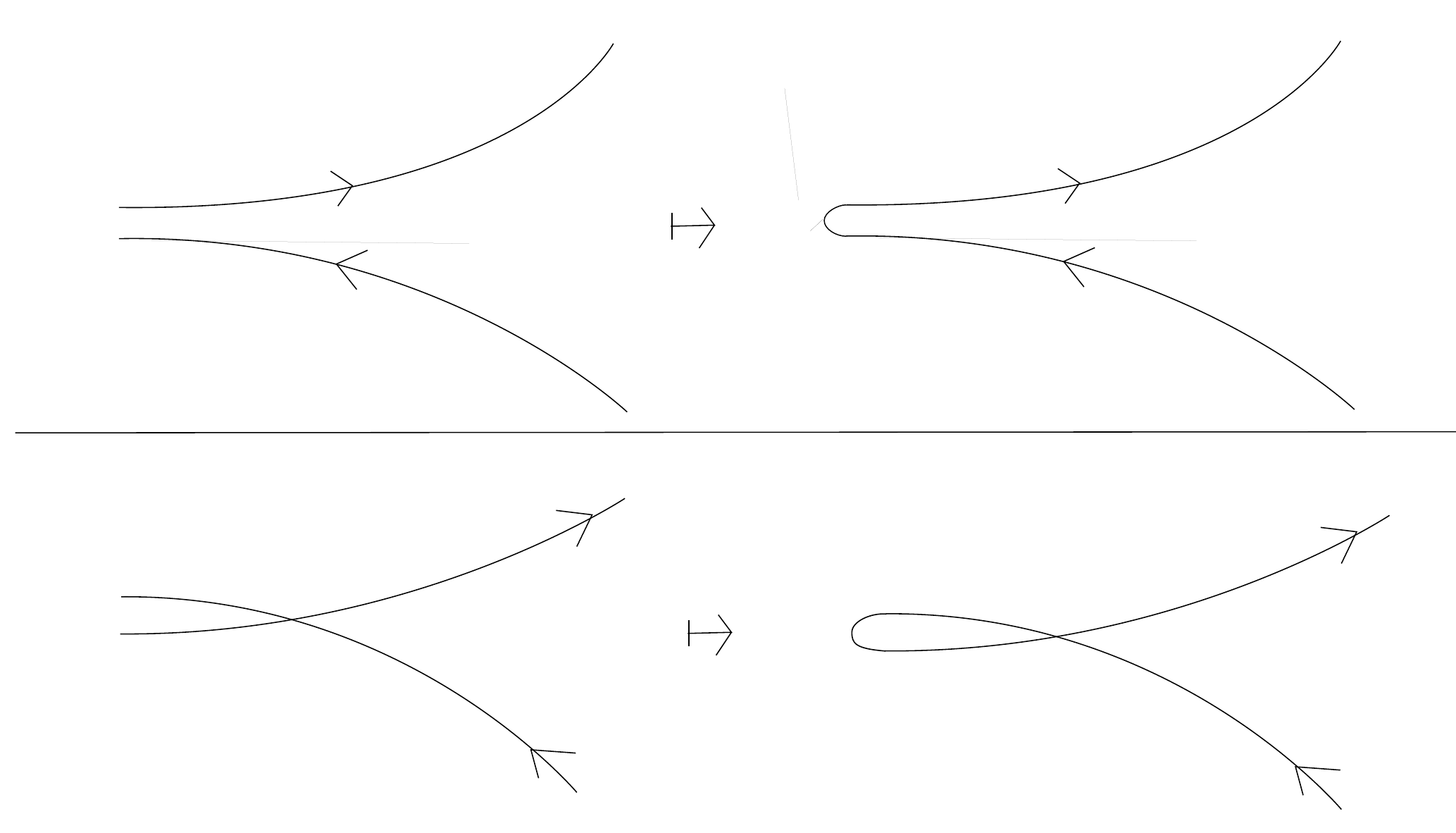}
    \caption{two ways of composing curves}
    \label{fig:composing}
\end{figure}

(Regularized) holonomies are multiplicative under composition of paths:

\begin{Th}
For composable paths $\gamma_1$ and  $\gamma_2$, we have a product formula for the corresponding (regularized) holonomies:
\begin{equation}\label{eq:composable}
{\rm Hol}^{\rm reg}(A, \gamma_2 \gamma_1) = \holr(A, \gamma_2) \holr(A, \gamma_1).
\end{equation}
\end{Th}

\begin{proof}
We have,
$$
\holr(A_z,\gamma_1):=\Psi_{q_1}^{-1}(z)\Psi_{p_1}(z), \hskip 0.3cm
\holr(A_z,\gamma_2):=\Psi_{q_2}^{-1}(z)\Psi_{p_2}(z).
$$
By choosing the same point $z$ in both equations ({\em e.g.} in the small neighborhood of $\gamma_1(1) = \gamma_2(0)$), we obtain
$$
\holr(A_z,\gamma_2)\holr(A_z,\gamma_1) = \Psi_{q_2}^{-1}(z)\Psi_{p_1}(z) = {\rm Hol}^{\rm reg}(A, \gamma_2 \gamma_1),
$$
as required.
\end{proof}

% We consider a complex plane $\mathbb{C}$ with $n$ distinct marked points $z_1, \dots, z_n$, and we denote
% %
% $$
% \Sigma= \mathbb{C}\backslash \{ z_1, \dots, z_n\}.
% $$

% Associated with the surface $\Sigma$, we define the regular path groupoid ${\rm Path}_n$: its objects are pairs $(z_i, v_i)$, where $z_i$ is one of the marked points and $v_i \in T_{z_i}\mathbb{C}$ is a tangent vector at that point.
% The morphisms are regular homotopy classes (that is, the first Reidemeister move is forbidden) of paths $\gamma: [0,1] \to \mathbb{C}$ with $\dot{\gamma}(t) \neq 0$ and
% %
% $$
% \gamma(0)=z_i, \hskip 0.3cm \dot{\gamma}(0)=v_i,
% \hskip 0.3cm \gamma(1)=z_j, \hskip 0.3cm \dot{\gamma}(1)=-v_j,
% $$
% and regular homotopies preserve derivatives at the end points of the path. 

% The framing of the surface is a trivialization of its tangent bundle $T\Sigma \simeq \mathbb{R}^2\times \Sigma$, that is to say we smoothly give two linearly independent vectors at each point on surface, the blackboard framing is the one we fix the tangent fields $\frac{\partial}{\partial x}$ and $\frac{\partial}{\partial y}$, we use the blackboard framing for the rest of the paper.

For future use, we will need a notion of rotation number of a path. It is only defined for paths up to regular homotopy. That's why we will be using it for paths $\gamma$ with both end points being tangential end points. Hence, we have 
$$
(\gamma(0), \dot{\gamma}(0)) = (z_i, v_i), \hskip 0.3cm
(\gamma(1), \dot{\gamma}(1)) = (z_j, -v_j)
$$
for some marked points $z_i, z_j$ and $v_i=\rho_i e^{i \varphi_i},
v_j = \rho_j e^{i\varphi_j}$. We compute,
\begin{equation}       \label{eq:rot}
\int_0^1 d\log(\dot{\gamma}(t)) = \log \ \left| \frac{v_j}{v_i} \right| + 2 \pi i \ \rm{rot}_{\gamma},
\end{equation}
where the imaginary part of this expression is the rotation number of $\gamma$. In particular, we have
$$
\rm{rot}_{\gamma} \in \mathbb{Z} + \frac{1}{2} + \frac{1}{2\pi}(\varphi_j - \varphi_i).
$$
%
%The rotation number with respect to the blackboard framing is given by formula
%\begin{equation}\label{eq:rotation_number}
%\rm{rot}_{\gamma}=\frac{1}{2\pi i}\int_{0}^{1}d\log(\dot{\gamma}(t)).
%\end{equation}
% For any regular path $\gamma:[0,1]\to \C,\dot{\gamma}(t)\ne 0$ of the path groupoid. The blackboard framing induce maps $(\dot{\gamma}(t),\gamma(t))\xrightarrow{f=(f_1,f_2)} \R^2\times \Sigma$,
% and we have
% \begin{equation}\label{eq:framing_curve}
% \dot{\gamma}:[0,1]\to T_{\gamma(t)}\Sigma_{0,n}\xrightarrow{f_1}\R^2\setminus \{0\}
% \end{equation}
% then we can define the rotation number of the curve $\gamma$ to be the winding number of $\dot{\gamma}$,
% \begin{equation}\label{eq:rotation_number}
% \rm{rot}_{\gamma}=\frac{1}{2\pi i}\int_{0}^{1}d\log(\dot{\gamma}(t)).
% \end{equation}
Note that ${\rm rot}_\gamma \in \mathbb{R}$, and that one full turn in the anticlockwise direction corresponding to ${\rm rot}_\gamma=1$. For the composition of paths, we obtain
$$
{\rm rot}_{\gamma_2 \gamma_1} =
{\rm rot}_{\gamma_1} + {\rm rot}_{\gamma_2}- \frac{1}{2},
$$
where $-1/2$ on the right hand side comes from the extra clockwise half-turn in the definition of the path composition $\gamma_2\gamma_1$.

Note that one can consistently choose $v=1$ in all tangential base points.  In that case, all rotation numbers of paths are half-integers:
${\rm rot}_\gamma \in \mathbb{Z} + 1/2$. Furthermore,  the regular fundamental group $\pi^{\rm reg}_1(\Sigma, (z_i, v_i))$ surjects onto the ordinary fundamental group $\pi_1(\Sigma, z_i)$.

% \begin{Th}
% Regularized holonomies are group-like elements of $U\mathfrak{f}_n$, and they define a representation of the groupoid ${\rm Path}_n$. That is, for composable paths $\gamma_1, \gamma_2$, we have
% %
% \begin{equation}\label{eq:composable}
% {\rm Hol}^{\rm reg}(A, \gamma_2 \gamma_1) = \holr(A, \gamma_2) \holr(A, \gamma_1).
% \end{equation}
% \end{Th}

% \begin{proof}
% For the first statement, by Lemma \ref{lem:solutions} local solutions are group like. Hence, so are regularized holonomies obtained by formula \eqref{eq:def_regularized_hol}. 

% For the second statement, denote by $(z_i, v_i)$ and $(z_j, v_j)$ the tangential base points corresponding to the start and end of $\gamma_1$, and $(z_j, v_j)$ and $(z_k, v_k)$ to the start and end of $\gamma_2$. Then,
% %
% $$
% \holr(A_z,\gamma_1):=\Psi_{z_j,v_j}^{-1}(z)\Psi_{z_i,v_i}(z), \hskip 0.3cm
% \holr(A_z,\gamma_2):=\Psi_{z_k,v_k}^{-1}(z)\Psi_{z_j,v_j}(z).
% $$
% By choosing the same point $z$ in both equations ({\em e.g.} in the small neighborhood of $z_j$), we obtain
% %
% $$
% \holr(A_z,\gamma_2)\holr(A_z,\gamma_1) = \Psi_{z_k,v_k}^{-1}(z)\Psi_{z_i,v_i}(z) = {\rm Hol}^{\rm reg}(A, \gamma_2 \gamma_1),
% $$
% as required.

% %When defining the composition of rotation number in subsection \ref{lem:solutions}, we do some small surgery to connect tangent vectors of composable curves, this is compatible with \eqref{eq:composable} because this surgery is operated on a small contractible region which does not affect the holonomy. Then \eqref{eq:composable} follows from the multiplication and cancellation of solutions.
% \end{proof}

\section{Proof of the generalized pentagon equation}\label{sec:Generalized_Pentagon_Equation}

In this Section, we prove Theorem \ref{thm:main}. We adapt Drinfeld's original proof to the case of many marked points and to paths which may have self-intersections.

\subsection{Asymptotic regions and local solutions}
\label{sub:regions_local_solution}
We consider the KZ connection for two moving points $z$ and $w$ in $\Sigma = \C \setminus \{ z_1, \dots, z_n \}$ with values in $\mathfrak{t}_{n+2}$:
\begin{equation}\label{eq:A_zw}
A_{z,w}= \frac{1}{2\pi i}\left(\sum_{k=1}^n (t_{k,z} d\log(z-z_k)
+ t_{k,w}d\log(w-z_k)) + t_{z,w} d\log(z-w)\right),
\end{equation}
and the corresponding KZ equation
\begin{equation}  \label{eq:KZzw}
    d\Psi = A_{z,w}\Psi.
\end{equation}

Let $\gamma: (0,1) \to \Sigma = \mathbb{C} \backslash \{ z_1, \dots, z_n\}$ be a path with a tangential starting point $(z_i, v_i)$, a tangential end point $(z_j, v_j)$, and a finite number of transverse self-intersections $a_l=\gamma(s_l) = \gamma(t_l)$ for $l=1, \dots, m$. As before, we assume that the path $\gamma$ is linear near end points and near self-intersection points.

The triangle  $\{(t,s)|0\le t\le s\le 1\}\subset \mathbb{R}^2$ in the $t-s$ plane is mapped to the set of ordered pairs of points on $\gamma$: 
$$
\gamma^{(2)}: (t,s) \mapsto (z=\gamma(t), w=\gamma(s)).
$$
Following Drinfeld's idea (see \cite{Drinfeld1990}, and also \cite{Kassel1995}), we evaluate regularized holonomies on different parts of the closed path shown on Fig. \ref{fig:path}. Since this path is contractible, the product of regularized holonomies is equal to 1, and this identity will give rise to the generalized pentagon equation.

% the higher order derivative of $\gamma$ vanishes at 0 and 1, more precisely near the end points we assume that
%\begin{align*}
%\gamma(\epsilon)=\gamma(0)+\epsilon\dot{\gamma}(0),\quad \gamma(1-\epsilon)=\gamma(1)-%\epsilon\dot{\gamma}(1),
%\end{align*}
%for small $\epsilon$.
%Inspired by Drinfeld's proof of the pentagon equation (see \cite{Drinfeld1990} and also \cite{Kassel1995}), we define solutions of the equation
%
%\begin{equation}   \label{eq:KZzw}
	%d\Psi = A_{z,w} \Psi
%\end{equation}
%in different asymptotic regions, and then evaluate regularized holonomies on different parts of the closed path shown on Fig. \ref{fig:path}. Since this path is contractible, we arrive at the identity which will be our generalized pentagon equation.

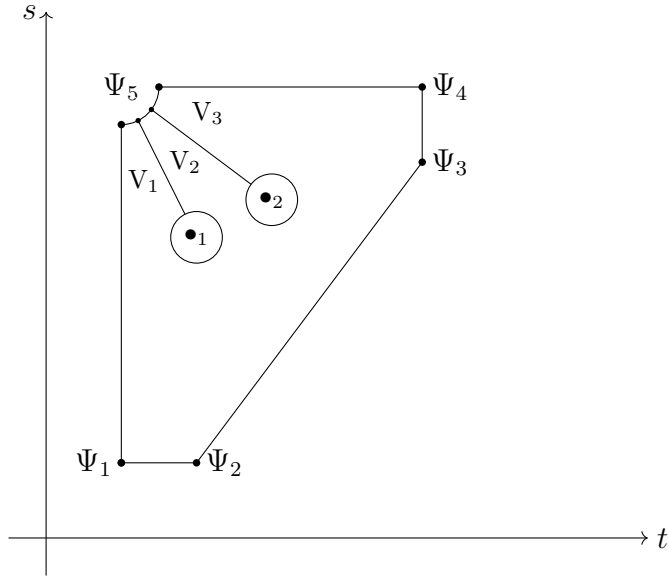
\begin{figure}[h]
	\begin{tikzpicture}
		\draw[->](0.5,0)--(9,0) node [right]{\large $t$};
		\draw[->](1,-0.5)--(1,7) node [left]{\large $s$};
		\draw(2,5.5)--(2,1)--(3,1)--(6,5)--(6,6)--(2.5,6);
		\draw (2,1) node [left] {\large $\Psi_1$};
		\draw (3,1) node [right]{\large $\Psi_2$};
		\draw (6,5) node [right]{\large $\Psi_3$};
		\draw (6,6) node [right]{\large $\Psi_4$};
		\draw (2,6) node {\large $\Psi_5$};
		\draw (2,1) node[circle,fill=black,inner sep=1pt]{};
		\draw (3,1) node[circle,fill=black,inner sep=1pt]{};
		\draw (6,5) node[circle,fill=black,inner sep=1pt]{};
		\draw (6,6) node[circle,fill=black,inner sep=1pt]{};
		\draw (2,1) node[circle,fill=black,inner sep=1pt]{};
		\draw (2,5.5) node[circle,fill=black,inner sep=1pt]{};
		\draw (2.5,6) node[circle,fill=black,inner sep=1pt]{};
		%%%%%%%%%%%%%
		\draw[white, name path=A] (2,6)--(4,4.5);
		\draw[white,name path=B] (2,6)--(3,4);
		\draw[name path=C] (2,5.5) arc (270:360:0.5);
		\fill [name intersections={of=A and C, by={a}}]
		(a) circle (1pt);
		\fill [name intersections={of=B and C, by={b}}]
		(b) circle (1pt);
		\draw (a)--(4,4.5) node [name path=D, circle,draw,fill=white,inner sep=3pt]{$\bullet_2$};
		\draw (b)--(3,4)node [name path=E, circle,draw,fill=white,inner sep=3pt]{$\bullet_1$};
  
        \node at (2.3,4.75) {$\mathrm{V}_1$};
        \node at (2.85,5.0) {$\mathrm{V}_2$};        
        \node at (3.15,5.65) {$\mathrm{V}_3$};
	\end{tikzpicture}
	\caption{$\Psi_1,\dots,\Psi_5$ are 5 solutions of the KZ equation and $\bullet_1,\bullet_2$ are two self intersection points of the curve in the (t,s) plane}
	\label{fig:path}
\end{figure}

In order to define regularized holonomies, we consider five asymptotic regions of the KZ equation associated to tangential base points:
%corresponding to 5 different limiting behaviors of $(s,t)$,
\begin{align}\label{eq:zone}
& 1) \hskip 0.3cm ((z_i z) w) z_j \hskip 0.5cm   0 < t \ll s \ll 1;  \notag\\   
%s\to 0,\quad\frac{t}{s}\to 0; \quad\quad  
& 2) \hskip 0.3cm (z_i(zw))z_j \hskip 0.5cm  0 < s-t \ll s \ll 1; \notag\\
& 3) \hskip 0.3cm  z_i ((zw)z_j) \hskip 0.5cm 0< s-t \ll 1-t \ll 1; \\
& 4) \hskip 0.3cm  z_i (z (wz_j)) \hskip 0.5cm 0 < 1-s \ll 1-t \ll 1; \notag\\
& 5) \hskip 0.3cm  (z_iz)(wz_j) \hskip 0.5cm 0 < t \ll 1, 0< 1-s \ll 1, \notag
%&3)1-t\to 0, \quad\frac{s-t}{1-t}\to 0;\\
%&2)s\to 0,\quad \frac{s-t}{s}\to 0;\quad\quad &4)1-t\to 0,\quad\frac{1-s}{1-t}\to 0\\
%&5) t\to 0,\quad 1-s\to 0;
\end{align}
Here paranthesized expressions of the type $((z_i z) w) z_j$ label tangential base points and the corresponding asymptotic regions.

%In the following, we first consider the solutions around these 5 distinguished zone and then the self intersection zone.

%\subsection{Local solutions} \label{sub:first_region}

Consider the first asymptotic region. It is convenient to introduce the new coordinates $(u,v)$ by
\[
u:=\frac{z-z_i}{w-z_i}, \quad v:=\frac{w-z_i}{v_i}.
\]
Let
\[
 D_{\varepsilon} := \{ (u,v) \in B_{\varepsilon}(0,0) \mid u \notin \R_{<0}, v \notin v_i\R_{<0} \}
\]
be a small polydisc around $(0,0)$ in $(u,v)$ coordinates with two branch cuts removed, so that
$\log(u)$ and $\log(v)$ (taking real values on the rays opposite to branch cuts) are analytic functions on $D_\varepsilon$. Note that the first asymptotic region $0<t \ll s \ll 1$ maps to $D_\varepsilon$ by the map $\gamma^{(2)}$. Recall the following fact:

\begin{Lemma}\label{lem:existence_sol}
For $\varepsilon$ sufficiently small, there is a unique solution $\Psi_1$ of the KZ equation on $D_{\varepsilon}$ such that
\begin{align*}
\Psi_1(z,w)=f(u,v) \left(\frac{z-z_i}{v_i}\right)^{\frac{t_{iz}}{2\pi i}}\left(\frac{w-z_i}{v_i}\right)^{\frac{t_{iw}+t_{zw}}{2\pi i}},
%\exp\left(\log\left(\frac{z-z_i}{v_i}\right)\frac{t_{iz}}{2\pi i}\right)\exp\left(\log\left(\frac{w-z_i}{v_i}\right)\frac{(t_{iw}+t_{zw})}{2\pi i}\right)
\end{align*}
where $f(u,v)$ is an analytic function with $f(0,0)=1$.	
\end{Lemma}

\begin{proof}
See Section 2 in \cite{Drinfeld1990}. 
%Also by a similar proof as in \ref{self-intersection}, the solution exists in $D_{\epsilon,1}(u,v)$ and $f$ is analytical in $B_{\epsilon,1}(u,v)$.
\end{proof}

\begin{Rem}
We use the following shorthand notation for asymptotic behavior of the function $\Psi_1(z,w)$:
\begin{equation}
\Psi_1(z,w)
	\sim_{((z_iz)w)z_j}\left(\frac{z-z_i}{v_i}\right)^{\frac{t_{iz}}{2\pi i}}\left(\frac{w-z_i}{v_i}\right)^{\frac{t_{iw}+t_{zw}}{2\pi i}}.
\end{equation}
Similarly, in the other four asymptotic regions we get the following solutions:
\begin{align}
&\Psi_2(z,w)\sim_{\left(z_i\left(z w\right)\right)z_j}\left( \frac{w-z}{v_i}\right)^{\frac{t_{zw}}{2\pi i}} \, \left(\frac{w-z_i}{v_i}\right)^{\frac{t_{iz}+t_{iw}}{2\pi i}},\\
&\Psi_3(z,w)\sim_{z_i\left(\left(zw\right) z_j\right)} \left(\frac{w-z}{-v_j}\right)^{\frac{t_{zw}}{2\pi i}} \, \left(\frac{z-z_j}{v_j}\right)^{\frac{t_{zj}+t_{wj}}{2\pi i}},\\
&\Psi_4(z,w)\sim_{z_i\left(z\left(wz_j\right)\right)} \left(\frac{w-z_j}{v_j}\right)^{\frac{t_{wj}}{2\pi i}}\left(\frac{z-z_j}{v_j}\right)^{\frac{t_{zw}+t_{zj}}{2\pi i}},\\
&\Psi_{5}(z,w)\sim_{(z_iz)(wz_j)}\left(\frac{z-z_i}{v_i}\right)^{\frac{t_{iz}}{2\pi i}}\left(\frac{w-z_j}{v_j}\right)^{\frac{t_{wj}}{2\pi i}}.
\end{align}
Note that in the expressions for $\Psi_2$ and $\Psi_3$ we are using tangential base points $(0,v_i)$ and $(0, v_j)$ for the $(zw)$ configuration.
\end{Rem}

\subsection{Regularized holonomies}
\label{sub:self_intersection_contributions}
In this Section, we compute regularized holonomies for different parts of the closed path of Fig. \ref{fig:path}. More precisely, we extend the asymptotic solutions $\Psi_1, \dots, \Psi_4$ from above to the simply connected domain shown in Fig \ref{fig:path}. For $\Psi_5$ we obtain solutions
\[
\Psi_{5,l}, \quad l = 1, \dots, m+1,
\]
where $\Psi_{5,l}$ coincides with the solution $\Psi_5$ in the region $\mathrm{V}_l$. Moreover, for each intersection point $l=1,\dots,m$ we will define two solutions $\Psi_{a_l,u},$ and $\Psi_{a_l,d}$ related by
\[
\Psi_{a_l,d} = \Psi_{a_l,u} \exp(-\epsilon_l t_{z,w}).
\]
With that the generalized pentagon equation will follow by identifying the terms in
\begin{multline*}
\underbrace{\Psi_4^{-1} \Psi_3}_{\Phi_{\rm{KZ}} \atop\text{see \eqref{eq:Phi2}}} 
\underbrace{\Psi_3^{-1} \Psi_2}_{H_{zw} \left| \frac{v_j}{v_i} \right| ^{\frac{t_{zw}}{2\pi}} e^{t_{zw} {\rm rot}_\gamma} \atop\text{see \eqref{eq:H_zw}}}
\underbrace{\Psi_2^{-1} \Psi_1}_{\Phi_{\rm{KZ}} \atop\text{see \eqref{eq:Phi1}}} =  \\ 
=
\underbrace{\Psi_4^{-1}\Psi_{5,m+1}}_{H_z \atop\text{see \eqref{eq:H_z}}} \prod_{l=1}^m \left(\underbrace{\Psi_{5,l+1}^{-1} \Psi_{a_l,u}}_{C_l^{-1}}
\underbrace{\Psi_{a_l,u}^{-1} \Psi_{a_l,d}}_{e^{-\epsilon_i t_{zw}} \atop\text{see \eqref{eq:intterms}}}
\underbrace{\Psi_{a_l,d}^{-1} \Psi_{5,l}}_{C_l} \right)
\underbrace{\Psi_{5,1}^{-1} \Psi_1}_{H_w \atop\text{see \eqref{eq:H_w}}}.
\end{multline*}

We are carrying out these identificatons in the following sections. The proof of $H_{z},H_{w}, H_{zw}$ and $\Phi_{\rm{KZ}}$ contributions are similar to Section 2 of Drinfeld \cite{Drinfeld1990} and Lemma XIX.8.2 in \cite{Kassel1995}

\subsubsection{Self-intersection points} \label{sub:self_intersection}

In this Section, we consider solutions of \eqref{eq:KZzw} near self-intersection points. Suppose that the path $\gamma$ has $m$ transverse self-intersection points.  We order these points in the counter-clockwise direction by their positions on the $s-t$ plane, see Fig. \ref{fig:path}. 
If two self-intersection points are located on the same ray, we perturb the path by a small regular homotopy. For each self-intersection point, we denote by $\varepsilon_l$ the local intersection number and by  $a_l=\gamma(s_l)=\gamma(t_l)$ the position of the intersection point, where $s_l>t_l$, $l=1,\dots,m$.
By perturbing the curve $\gamma$ by a regular homotopy if needed, we can assume that $\gamma$ is a linear function for $t$ near $t_l$ and $s_l$ for $l=1, \dots, m$.

We denote by $L_l$ the line in the $z-w$ plane corresponding to the straight line 
\begin{align}
s=\frac{s_l-1}{t_l}t+1
\end{align}
in the $t-s$ plane which connects the points $(0,1
%[-\frac{(s_l-1)v_j}{t_lv_i}:1]
)$ and $(t_l,s_l)$. Since $w-z=\gamma(s)-\gamma(t)$, it is convenient to introduce notation
\begin{align}
u_l=\frac{d}{dt} \ \left(\gamma\left(\frac{s_l-1}{t_l}t+1\right)-\gamma(t)\right)|_{t=t_l}.
\end{align}
By assumptions, $\gamma(t)$ is linear in $t$ near $t_l$ and $s_l$ which implies
\begin{equation}
((z-w) \circ L_l)(t) = a_l+u_l(t-t_l),\quad \text{for}\quad  |t-t_l| \quad \text{sufficiently small}.
\end{equation}
%\begin{Rem}
%The assumption that $\gamma(t)$ is linear for $t$ near $t_l$
%amounts to flattening some parts of the curve by a regular homotopy equivalence.
%\end{Rem}

We denote by  $B_{\varepsilon,a_l}(z-a_l,w-z)$ the polydisk centered at $(0,0)$ with radius $\varepsilon$ and let $l_{a_l}=\{(z-a_l,w-z); w-z=\gamma(s)-\gamma(t),s=\frac{s_l-1}{t_l}t+1,0\le t\le t_l\}$ denote a branch cut of the function $\log(w-z)$, and we denote 
$$
D_{\varepsilon,a_l}:=B_{\varepsilon,a_l}(z-a_l,w-a_l)\setminus l_{a_l} \cap B_{\varepsilon,a_l}(z-a_l,w-a_l).
$$

Similar to Theorem 4.6 in \cite{Brown2009} (see also \cite{Deligne1970}), there is a unique solution $\Psi_{a_l, u}(z,w)$ of the KZ equation on $D_{\varepsilon,a_l}$ such that 
\begin{align*}
		\Psi_{a_l, u}(z,w)=f(z-a_l,w-z)\exp\left(\log\left(\frac{w-z}{-u_l}\right)\frac{t_{zw}}{2\pi i} +\epsilon_l \frac{t_{zw}}{2} \right)
\end{align*}
 where $f$ is analytic on  $B_{\varepsilon,a_l}$ and $f(0,0)=1$, where
 \begin{align}
 \epsilon_l = \mathrm{sign} \left( \mathrm{Im}\left( \frac{\dot{\gamma}(s_l)}{\dot{\gamma}(t_l)} \right) \right)
 \end{align}
 is the sign of the intersection. 
 
The normalization is chosen such that it satisfies the following property. If we extend $\Psi_{a_l,u}$ continuously over the line $L_l$ from above, then on the line we have
 \[
 \Psi_{a_l,u} \sim (t-t_l)^{t_{zw}},
 \]
 Similarly, we can define a solution $\Psi_{a_l,d}$ that has the desired asymptotics when we extend over $L_l$ from below, namely as analytic extension of
 \begin{align*}
		\Psi_{a_l, d}(z,w)=f(z-a_l,w-z)\exp\left(\log\left(\frac{w-z}{-u_l}\right)\frac{t_{zw}}{2\pi i} - \epsilon_l \frac{t_{zw}}{2} \right),
\end{align*}
 so that we have
$$
\Psi_{a_l, d}(z,w) = \Psi_{a_l, u}(z,w) \, e^{-\varepsilon_l t_{zw}},
$$
as well as
\[
\lim_{(z,w) \mathrel{\downarrow}{(z_0,w_0)}} \Psi_{a_l,u}(z,w) = \lim_{(z,w) \mathrel{\uparrow}{(z_0,w_0)}} \Psi_{a_l,d}(z,w).
\]
where $(z_0,w_0)$ is any point on the line $L_l$ and the limits are approaching the line from above and below, respectively.
Recall that $\Psi_{5,l}$ and $\Psi_{5,l+1}$ are defined by the condition that they coincide with the asymptotic solution at $(t,s)=(0,1)$ in the regions $\mathrm{V}_l$ and $\mathrm{V}_{l+1}$, respectively. Thus we have
\[
\lim_{(z,w) \mathrel{\downarrow}{(z_0,w_0)}} \Psi_{5,l+1}(z,w) = \lim_{(z,w) \mathrel{\uparrow}{(z_0,w_0)}} \Psi_{5,l}(z,w).
\]
We now define the regularized holonomy along $L_l$ to be
\begin{equation}
C_l=\holr(A_{z,w},L_l):=\Psi_{a_l, u}^{-1} \Psi_{5, l+1}.
\end{equation}
From the above limit considerations we obtain
\begin{align*}
    C_l &= \lim_{(z,w) \mathrel{\downarrow}{(z_0,w_0)}}   \Psi_{a_l, u}^{-1}(z,w) \Psi_{5, l+1}(z,w) \\
    &= \lim_{(z,w) \mathrel{\uparrow}{(z_0,w_0)}}   \Psi_{a_l, d}^{-1}(z,w) \Psi_{5, l}(z,w) \\
    &= \Psi_{a_l,d}^{-1} \Psi_{5,l+1}.
\end{align*}

%  The solution 
%  $\Psi_{a_l, d}$ is normalized on the down side of the branch cut. A similar solution normalized on the upper side of the cur is given by
%  %
%  $$
% \Psi_{a_l, u} = \Psi_{a_l, d} \, e^{\varepsilon_l t_{zw}}.
%  $$

% Similarly, we introduce solutions $\Psi_{5, l}$ for $l=1, \dots, m+1$ which are normalized between the cuts number $l-1$ and $l$. In particular, the solution
% $\Psi_{5, 1}$ is normalized below all the cuts, and $\Psi_{5, m+1}$ is normalized above all the cuts.

% With this notation,  the regularized holonomy of $L_l$ is defined by
% \begin{align*}
% C_l=\holr(A_{z,w},L_l):=\Psi_{a_l, u}^{-1}(z,w)\Psi_{5, l+1}(z,w),
% \end{align*}
% where $(z,w)$ is any point in the simply-connected open domain shown in Fig. \ref{fig:path}. 
% %Here $\Psi_{5, l}(z,w)$ is the solution $\Psi_5(z,w)$ normalized at the down side of the $l$'s branch cut. 
% We continuously extended $\Psi_{a_l,u}$ to the line $L_l$ so that we can choose $(z_0,w_0) \in L_l$ and obtain
% \[
% C_l= \Psi_{a_l, u}^{-1}(z_0,w_0) \lim_{(z,w) \mathrel{\downarrow}{(z_0,w_0)}} \Psi_{5, l+1}(z,w),
% \]
% where the limit is approaching from above the line. Note that $\Psi_{5,l}$ and $\Psi_{5,l+1}$ have the same asymptotics around $(t,s) = (0,1)$, and by definition $\Psi_{a_l,d}(z_0,w_0) = \Psi_{a_l,u}(z_0,w_0)$. From this we obtain
% \[
% C_l= \Psi_{a_l, d}^{-1}(z_0,w_0) \lim_{(z,w) \mathrel{\uparrow}{(z_0,w_0)}} \Psi_{5, l}(z,w).
% \]
We conclude that
\begin{align}\label{eq:intterms}
    \Psi_{5,l+1}^{-1} \Psi_{5,l} &= \Psi_{5,l+1}^{-1} \Psi_{a_l,u} \Psi_{a_l,u}^{-1} \Psi_{a_l,d} \Psi_{a_l,d}^{-1} \Psi_{5,l} \notag \\
    &= C_l^{-1} e^{-\epsilon_l t_{zw}} C_l.
\end{align}

% Let $\gamma_{a_l}$ denote the curve in the z-w plane corresponding to the counter-clock wise oriented  circle from one side of the branch cut to the other side. The corresponding regularized holonomy  is given by 
% \begin{align}
% \holr(A_{z,w},\gamma_{a_l})=\Psi_{a_l, u}(z,w)^{-1} \Psi_{a_l, d}(z,w) =e^{-\varepsilon_l t_{zw}},
% \end{align}
% Furthermore, we have
% %
% \begin{align}
%     C_l^{-1}=\holr(A_{z,w},\gamma_l^{-1}):=\Psi_{5, l+1}(z,w)^{-1}\Psi_{a_l, u}(z,w).
% \end{align}
% Putting things together, we obtain
% %
% \begin{align}
% \begin{array}{lll}
%    \holr(A_{z,w},\gamma_m^{-1}\gamma_{a_m}\gamma_m\dots \gamma_1^{-1}\gamma_{a_1}\gamma_1) & = & \Psi_{5, m+1}(z,w)^{-1} \Psi_{5, 1}(z,w) \\
%    & = &
%    \prod_{l=1}^{m}C^{-1}_{l}e^{-\varepsilon_it_{z,w}}C_{l}.
%    \end{array}
% \end{align}

\subsubsection{$H_z$ and $H_w$}

First, we consider the expression $\Psi_{5, 1}(z,w)^{-1}\Psi_1(z,w)$, where the two solutions are analytically continued in the strip of the variable $w$ along the path $\gamma$. 

At this point, it is convenient to introduce new functions
$$
\tilde{\Psi}_{5,1}(z,w):=\left(\frac{z-z_i}{v_i}\right)^{\frac{-t_{iz}}{2\pi i}}\Psi_{5,1}(z,w), \hskip 0.3cm
\tilde{\Psi}_{1}(z,w):=\left(\frac{z-z_i}{v_i}\right)^{\frac{-t_{iz}}{2\pi i}}\Psi_{1}(z,w).
%=\Psi_{5,1}(z,w)\big(\frac{z-z_i}{v_i})^{\frac{-t_{iz}}{2\pi i}}.
$$
Observe that 
$$
\Psi_{5, 1}(z,w)^{-1}\Psi_1(z,w)=\tilde{\Psi}_{5, 1}(z,w)^{-1} \tilde{\Psi}_1(z,w),
$$
and that functions $\tilde{\Psi}_{5,1}(z,w)$ and $\tilde{\Psi}_1(z,w)$ satisfy the same differential equation (the conjugate of the KZ equation by the factor $((z-z_i)/v_i)^{t_{iz}/2\pi i}$. Furthermore, both $\tilde{\Psi}_{5,1}(z,w)$ and $\tilde{\Psi}_1(z,w)$ are regular functions at $z=z_i$, 
and $\tilde{\Psi}_{5,1}(z_i,w)$ and $\tilde{\Psi}_1(z_i,w)$ are solutions of the KZ equation 
\begin{equation}
d\tilde{\Psi}=
\left(\frac{t_{iz,w}}{2\pi i} \, d\log(w-z_i)+\sum_{k\ne i}\frac{t_{k,w}}{2\pi i}\, d\log(w-z_k)\right)\tilde{\Psi}.
\end{equation}
In fact,
$$
\tilde{\Psi}_{5,1}(z_i,w)=\Psi_{(z_j, v_j)}(w), \hskip 0.3cm
\tilde{\Psi}_1(z_i,w) = \Psi_{(z_i, v_i)}(w).
$$
We can now conclude
\begin{equation} \label{eq:H_w}
\begin{array}{lll}
\Psi_{5, 1}(z,w)^{-1}\Psi_1(z,w) & = &\tilde{\Psi}_{5, 1}(z,w)^{-1} \tilde{\Psi}_1(z,w) \\
& = & \Psi_{(z_j, v_j)}(w)^{-1} \Psi_{(z_i, v_i)}(w) \\
& = & 
\holr\left( \frac{t_{iz,w}}{2\pi i} \, d\log(w-z_i)+\sum_{k\ne i}\frac{t_{k,w}}{2\pi i}\, d\log(w-z_k),\gamma\right)
%{\rm Hol}^{\rm reg}(A_{\rm KZ}^{iz, w}, \gamma) 
\\
& = & H_w.
\end{array}
\end{equation}
In a similar way, we have
\begin{equation} \label{eq:H_z}
\begin{array}{lll}
\Psi_4(z,w)^{-1} \Psi_{5, m+1}(z,w) & = &
\holr\left( \frac{t_{z,wj}}{2\pi i}\, d\log(z-z_j)+\sum_{k\ne j}\frac{t_{k,z}}{2\pi i}\, d\log(z-z_k),\gamma\right)\\
& = & H_z.
\end{array}
\end{equation}

% Note that the solutions $\Psi_{5, 1}(z,w)$ and $\Psi_{5, m+1}(z,w)$ are at the different sides of all branch cuts corresponding to self-intersection points of the path $\gamma$. Together with the results of the previous section, equations \eqref{eq:H_w} and \eqref{eq:H_z} imply
% %
% \begin{equation}      \label{eq:key_up}
% \Psi_4(z,w)^{-1}\Psi_1(z,w) = H_z  \left(\prod_{l=1}^{m}C^{-1}_{l}e^{-\varepsilon_it_{z,w}}C_{l} \right) H_w.
% \end{equation}

\subsubsection{$H_{zw}$}

In this Section, we consider the expression $\Psi_3(z,w)^{-1} \Psi_2(z,w)$. 

To start with, it is useful to consider the following simple differential equation
$$
d\phi = t_{z,w} d\log(z-w) \phi.
$$
We are interested in solutions of this equation for $z=\gamma(t), w=\gamma(s)$ with $t<s$. It convenient to introduce two normalized solutions $\phi_i(z,w)$ and $\phi_j(z,w)$,
\begin{equation}
\lim_{(s,t) \to (0,0)} \phi_i(z,w) \left(s-t\right)^{-\frac{t_{zw}}{2 \pi i}}  = 1,
\hskip 0.3cm
\lim_{(s,t) \to (1,1)} \phi_j(z,w) \left(s-t\right)^{-\frac{t_{zw}}{2 \pi i}}  = 1.
\end{equation}
The normalizations are chosen such that 
\begin{align*}
    \phi_i(z,w) &= \left(\frac{w-z}{v_i}\right)^{\frac{t_{zw}}{2 \pi i}}, &
    \phi_j(z,w) &= \left(\frac{w-z}{-v_j}\right)^{\frac{t_{zw}}{2 \pi i}}
\end{align*}
in the second and third asymptotic region, respectively.

Note that expressions $\phi_i(z,w) \left(s-t\right)^{-\frac{t_{zw}}{2 \pi i}}$ and $\phi_j(z,w) \left(s-t\right)^{-\frac{t_{zw}}{2 \pi i}}$ admit continuous extensions to the line  $s=t$, we denote these extensions by $\varphi_i(t)$ and $\varphi_j(t)$, respectively. These functions are solutions 
of the following  differential equation:
\[
d \varphi = t_{zw} \, d \log\left( \lim_{s \to t} \frac{\gamma(s) - \gamma(t)}{s-t} \right) \, \varphi = t_{zw} \, d \log \dot{\gamma}(t) \, \varphi.
\]
% \[
% d- t_{zw} d \log\left( \lim_{s \to t} \frac{\gamma(s) - \gamma(t)}{s-t} \right) = d- t_{z,w} d \log \dot{\gamma}(t).
% \]
This implies
\begin{align*}
\phi_j(z, w) \phi_i^{-1}(z,w) &=  \phi_j(z,w) \left(s-t\right)^{-\frac{t_{zw}}{2 \pi i}} \left( \phi_i(z,w) \left(s-t\right)^{-\frac{t_{zw}}{2 \pi i}} \right)^{-1} \\
&= \exp \left( t_{zw} \int_0^1 d \log \dot{\gamma}(t) \right) \\
&= \left| \tfrac{v_j}{v_i} \right|^{\frac{-t_{zw}}{2\pi}}  e^{-t_{zw} {\rm rot}_\gamma},
\end{align*}
where we used definition \eqref{eq:rot} in the last equation. Similarly to the previous Section, we define
$$
\tilde{\Psi}_2(z,w)=\phi_i^{-1}(z,w)\Psi_2(z,w), 
\hskip 0.3cm
\tilde{\Psi}_3(z,w) = \phi_j^{-1}(z,w)\Psi_3(z,w).
$$

Again, $\tilde{\Psi}_2(z,w)$ and $\tilde{\Psi}_3(z,w)$ satisfy the same differential equation (the conjugate of the KZ equation by the factor $(w-z)^{t_{zw}/2\pi i}$), they are regular at $w=z$, and their values $\tilde{\Psi}_2(z=w)$ and $\tilde{\Psi}_3(z=w)$ satisfy 
the equation
$$
d\tilde{\Psi}=\left( \sum_{k=1}^n \frac{t_{k,zw}}{2\pi i} \, d\log(z-z_l)\right) \tilde{\Psi}.
$$
By considering the asymptotic behavior, we identify
$$
\tilde{\Psi}_2(z=w) = \Psi_{(z_i, v_i)}^{1, \dots, n, zw}(z=w),
\hskip 0.3cm
\tilde{\Psi}_3(z=w) = \Psi_{(z_j, v_j)}^{1, \dots, n, zw}(z=w)
$$.
Using that $[t_{zw}, t_{i, zw}]=0$ for all $i=1, \dots, n$, we conclude
\begin{equation}      \label{eq:H_zw}
\begin{array}{lll}
\Psi_3(z,w)^{-1} \Psi_2(z,w) & =  &
\tilde{\Psi}_3(z=w)^{-1}  
\left| \frac{v_j}{v_i} \right|^{\frac{t_{zw}}{2\pi}} 
e^{t_{zw} {\rm rot}_\gamma} \, \tilde{ \Psi}_2(z=w) \\
& = & \tilde{\Psi}_3(z=w)^{-1} \tilde{ \Psi}_2(z=w) \, \left| \frac{v_j}{v_i} \right| ^{\frac{t_{zw}}{2\pi}} e^{t_{zw} {\rm rot}_\gamma} \\
& = & {\rm Hol}^{\rm reg}\left( \sum_{k=1}^n \frac{t_{k,zw}}{2\pi i} \, d\log(z-z_l), \gamma\right) \left| \frac{v_j}{v_i} \right| ^{\frac{t_{zw}}{2\pi}} e^{t_{zw} {\rm rot}_\gamma}  \\
& = & H_{zw} \left| \frac{v_j}{v_i} \right| ^{\frac{t_{zw}}{2\pi}} e^{t_{zw} {\rm rot}_\gamma}.
\end{array}   
\end{equation}

\subsubsection{$\Phi_{\rm KZ}$ contributions}

Next, we consider the regularized holonomy $\Psi_2(z,w)^{-1} \Psi_1(z,w)$. In the asymptotic region $0 < s,t \ll 1$, it is convenient to introduce
$$
\tilde{\Psi}_{1,2}(z,w) = \left( \frac{w-z_i}{v_i} \right)^{-(t_{iz} + t_{iw}+t_{zw})/2\pi i} \Psi_{1,2}(z,w).
$$
On the one hand, we have
$$
\tilde{\Psi}_2(z,w)^{-1} \tilde{\Psi}_1(z,w) = \Psi_2(z,w)^{-1} \Psi_1(z,w).
$$
And on the other hand, $w=z_i$ is not a singularity of the differential equation, $\tilde{\Psi}(w=z_i)$ satisfies the following equation:
$$
d \tilde{\Psi} = \left( t_{iz} d\log\left(\frac{z-z_i}{w-z_i} \right) +
t_{zw} d\log\left(\frac{w-z}{w-z_i} \right) + \dots \right) \tilde{\Psi},
$$
where $\dots$ stand for the terms regular in $(z-z_i), (w-z_i)$. By making the change of variables
$$
\zeta = \frac{z-z_i}{w-z_i}, \hskip 0.3cm
1-\zeta = \frac{w-z}{w-z_i}
$$
we identify the regularized holonomy with the Drinfeld associator
\begin{equation} \label{eq:Phi1}
    \Psi_2(z,w)^{-1} \Psi_1(z,w) = \Phi_{\rm KZ}(t_{iz}, t_{zw}).
\end{equation}
In a similar fashion, we obtain
\begin{equation} \label{eq:Phi2}
    \Psi_4(z,w)^{-1} \Psi_3(z,w) = \Phi_{\rm KZ}(t_{zw}, t_{wj}).
\end{equation}

% Putting together equations \eqref{eq:H_zw}, \eqref{eq:Phi1}, and \eqref{eq:Phi2}, we obtain
% %
% \begin{equation}      \label{eq:key_down}
%     \Psi_4(z,w)^{-1} \Psi_1(z,w) = \Phi_{\rm KZ}(t_{zw}, t_{wj}) H_{zw} \left| \frac{v_j}{v_i} \right|^{t_{zw}/2\pi i} e^{t_{z,w} {\rm rot}_\gamma} \Phi_{\rm KZ}(t_{iz}, t_{zw}).
% \end{equation}
% Comparison of this equation with equation \eqref{eq:key_up} finishes the proof of the generalized pentagon equation.

% I feel it is better to set the solution $\Psi_2$ as follows, other the equation will not have the associator
%\begin{align*}
%\Psi_2(z,w)\sim_{\left(z_i\left(z w\right)\right)z_j}(\frac{w-z}{v_i})^{\frac{t_{zw}}{2\pi i}} \, \left(\frac{w-z_i}{v_i}\right)^{\frac{t_{iz}+t_{iw}}{2\pi i}},\\
%\end{align*}
%If we put in this form of $\Psi_1$ and $\Psi_2$, then we multiply both $\Psi_1$ and $\Psi_2$ by $\left(\frac{w-z_i}{v_i}\right)^{(-t_{iz}-t_{iw}-t_{zw})/2\pi i}$, then they will become
%\begin{align*}
%&\rm{new}\Psi_1\sim \big(\frac{z-z_i}{w-z_i}\big)^{t_{iz}}\quad \rm{new}\Psi_2\sim (\frac{w-z}{w-z_i})^{t_{zw}}\\
%\end{align*}
%if we let $t=\frac{z-z_i}{w-z_i}$, new $\Psi_1$ and new $\Psi_2$ are the solutions of the following equation, so they differ by an classical associator $\Phi_{\rm{KZ}}(t_{iz},t_{zw})$.
%\begin{equation*}
%\Psi'(t)=\frac{1}{2\pi i}(\frac{t_{iz}}{t}+\frac{t_{zw}}{1-t})\Psi(t)
%\end{equation*}

\subsection{Properties of $C_l$}
\label{sub:property_C_l}
We now have a rather explicit description of all the terms in the generalized pentagon equation with the exception of holonomies $C_l$. In this Section, we establish an important property of these regularized holonomies.

Consider the quotient of the Lie algebra $\mathfrak{t}_{n+2}$ by the Lie ideal generated by $t_{zw}$. Denote the quotient Lie algebra by $\tau= \mathfrak{t}_{n+2}/\langle t_{zw} \rangle$ and the canonical projection by $\pi: \mathfrak{t}_{n+2} \to \tau$.  Observe that the image of generators $\pi(t_{i,z})$ and $\pi(t_{i,w})$ for $i=1, \dots, n$ span two commuting free Lie algebras with $n$ generators, that is $\tau \cong \mathfrak{f}_z \oplus \mathfrak{f}_w$.  Then the image of the connection $A_{zw}$ splits into two parts:
$$
%\begin{array}{lll}
\pi(A_{z,w}) = \frac{1}{2\pi i}\left(\sum_{k=1}^n (\pi(t_{k,z}) d\log(z-z_k) + \pi(t_{k,w})d\log(w-z_k))\right) = \pi(A_z) + \pi(A_w),
%\end{array}
$$
where $\pi(A_z)$ is a connection on $\Sigma = \mathbb{C}\backslash \{ z_1, \dots, z_n\}$ corresponding to the point $z$, and similarly $\pi(A_w)$ is a connection on $\Sigma$ corresponding to the point $w$.

\begin{Th}\label{thm:C_l}
    We have,
    \begin{equation}
        \pi(C_l) = {\rm Hol}^{\rm reg}(\pi(A_z), \gamma_{[0,t_l]}) 
        {\rm Hol}^{\rm reg}(\pi(A_w), \gamma_{[1,s_l]}).
    \end{equation}
\end{Th}

\begin{proof}
Recall that
$$
C_l=\Psi_{a_l, d}^{-1}(z,w)\Psi_{5, l}(z,w),
$$
where 
$$
\Psi_{a_l, d}(z,w) \sim \left(\frac{w-z}{-u_l}\right)^{t_{zw}/2\pi i}, \hskip 0.3cm
\Psi_{5, l}(z,w) \sim \left(\frac{z-z_i}{v_i}\right)^{t_{iz}/2\pi i}
\left(\frac{w-z_j}{v_j}\right)^{t_{jz}/2\pi i}.
$$
After taking a quotient, the self-intersection point $a_l$ becomes a regular (and not tangential) base point for $\pi(A_{zw})$, and we also have
$$
\pi(\Psi_{a_l, d})(a_l,a_l) = 1.
$$
Then,
$$
\pi(\Psi_{a_l, d}) = \Psi_{a_l}(z) \Psi_{a_l}(w),
$$
where $\Psi_{a_l}(z)$ and  $\Psi_{a_l}(w)$ are local solutions corresponding to connections $\pi(A_z)$ and $\pi(A_w)$ with the standard normalization at the regular point $\Psi_{a_l}(a_l)=1$.

Similarly, the asymptotic condition for $\Psi_{5, l}(z,w)$ splits into two independent factors:
$$
\pi(\Psi_{5, l})(z,w) \sim \left(\frac{z-z_i}{v_i}\right)^{\pi(t_{iz})/2\pi i}
\left(\frac{w-z_j}{v_j}\right)^{\pi(t_{jz})/2\pi i},
$$
and hence
$$
\pi(\Psi_{5, l})(z,w) = \Psi_{(z_i,v_i)}(z) \Psi_{(z_j, v_j)}(w).
$$

We conclude,
$$
\begin{array}{lll}
\pi(\Psi_{a_l, d}^{-1}(z,w)\Psi_{5, l}(z,w)) & = & \pi(\Psi_{a_l, d})^{-1}(z,w)\pi(\Psi_{5, l})(z,w)) \\
& = & (\Psi_{a_l}(z)^{-1} \Psi_{a_l}(w)^{-1})(\Psi_{(z_i,v_i)}(z) \Psi_{(z_j, v_j)}(w)) \\
& = & (\Psi_{a_l}(z)^{-1} \Psi_{(z_i,v_i)}(z))(\Psi_{a_l}(w)^{-1}\Psi_{(z_j, v_j)}(w)) \\
& = & {\rm Hol}^{\rm reg}(\pi(A_z), \gamma_{[0,t_l]}) 
        {\rm Hol}^{\rm reg}(\pi(A_w), \gamma_{[1, s_l]}),
\end{array}
$$
as required.   
\end{proof}

\end{document}